\documentclass[12pt]{amsart}
\usepackage{amsmath}
\usepackage{amsthm}
\usepackage{amssymb}
\usepackage{amscd}
\usepackage{bbold}
\usepackage[pdftex]{graphicx}
\usepackage{enumitem}
\usepackage{subfigure}

\newtheorem{lem}{Lemma}[section]
\newtheorem{thm}[lem]{Theorem}
\newtheorem{prop}[lem]{Proposition}

\theoremstyle{definition}
\newtheorem{defn}[lem]{Definition}
\newtheorem{example}[lem]{Example}

\newtheorem{rem}[lem]{Remark}

\newcommand{\CC}{{\mathbb C}}

\newcommand{\FF}{{\mathbb F}}

\newcommand{\QQ}{{\mathbb Q}}
\newcommand{\RR}{{\mathbb R}}

\newcommand{\TT}{{\mathbb T}}

\newcommand{\ZZ}{{\mathbb Z}}

\newcommand{\Qp}{\QQ_p}
\newcommand{\Zp}{\ZZ_p}

\newcommand{\Ccal}{\mathcal{C}}

\newcommand{\Ocal}{\mathcal{O}}

\DeclareMathOperator{\charact}{char}
\DeclareMathOperator{\ord}{ord}

\def\benm{\begin{enumerate}}            
\def\eenm{\end{enumerate}}              
\newcommand{\norm}[1]{\left\Vert #1\right\Vert}         

\title{ Frames of translates for number-theoretic groups  }

\date{\today}



\author{John J. Benedetto}
\address{Norbert Wiener Center\\
         Department of Mathematics \\
         University of Maryland \\
         College Park, MD 20742 \\
         USA}
\email{jjb@umd.edu}
\urladdr{http://www.math.umd.edu/\textasciitilde jjb}

\author{Robert L. Benedetto}
\address{Department of Mathematics and Statistics \\
Amherst College          \\
Amherst, MA        \\
        USA}
 \email{rlbenedetto@amherst.edu}
 \urladdr{https://rlbenedetto.people.amherst.edu/}

\thanks{The first-named author gratefully acknowledges the support
of ARO grant W911NF--17--1--0014
and NSF-ATD grant DMS-1738003 The second named author
gratefully acknowledges the support of NSF grant DMS-1501766.
The authors appreciate helpful comments by Carlos Cabrelli, Karlheinz Gr{\"o}chenig,
Eugenio Hern{\'a}ndez, and Victoria Paternostro. Finally, the authors are grateful for 
the constructive and thorough reviews by the two anonymous
referees. All of their suggestions have been incorporated into this final version.}

\begin{document}

\newcounter{bean}

\begin{abstract}
Frames of translates of $f \in L^2(G)$ are characterized in terms of the zero-set
of the so-called spectral symbol of $f$ in the setting of a locally compact abelian group
$G$ having a compact open subgroup $H$. We refer to such a $G$ as a number theoretic
group. This characterization was first proved in 1992 by Shidong Li and one of the authors
for $L^2({\mathbb R}^d)$ with the same {\it formal} statement of the characterization.
For number theoretic groups, and these include local fields, the strategy of proof is necessarily
entirely different; and it requires a new notion of translation that reduces to the usual definition in 
${\mathbb R}^d$.

\end{abstract}

\maketitle


\section{Introduction}

\subsection{Background}
\label{sec:background}

Time frequency analysis, wavelet theory, the theory of frames, sampling theory, and
shift invariant spaces have not only burgeoned, but have also uncovered a host
specific, tantalizing problem areas. One of these is the 
frame theoretic characterization of a closed span of translations. This is our bailiwick here, and
it has become a topic with great generalization, applicability, intricacy, and abstraction,
and with  a large number of contributors, see, e.g.,  \cite{BenLi-1993}, \cite{chri2016}, \cite{GroStr2007}, 
\cite{GolTou2008}, \cite{CabPat2010}, 
\cite{BowRos2015}, \cite{BarHerPat2015}, \cite{BarHerPat2016}, \cite{matu2018} and the
references therein.

We shall focus on the setting of what we call {\it number-theoretic 
LCAGs}, and by which we mean
locally compact abelian groups (LCAGs) $G$ with a compact open subgroup $H$.
For a given function $f$ on $G$, we shall solve the particular problem in this setting
of characterizing when the closed span of translates of $f$ is a frame. The characterization
is in terms of the zero-set of a natural spectral symbol,
see Theorem \ref{thm:char} for the solution. This {\it closed span of translates problem} is also addressed
in the aforementioned references but not for number theoretic groups. The Euclidean version,
going back to 1992, is restated in Theorem \ref{thm:benli}. The strategy for its
proof is natural, whereas the proof of Theorem \ref{thm:char} requires a new idea that we explain.

In Subsection \ref{sec:frames} we provide the necessary material on the theory of frames.
Then, in Subsection \ref{sec:Rdtranslates}, we state the Euclidean version of what we shall prove 
for number theoretic groups. Section \ref{sec:lcag} gives the theoretical background
for locally compact abelian groups and number theoretic groups that we need to prove
Theorem \ref{thm:char}.

In order to formulate the closed span of translates problem for number theoretic groups,
we require a new, motivated, and reasonable notion of translation. This is the content of
Subsection \ref{sec:trans}. Then, in Subsection \ref{sec:spanspecsym}, we define the
spectral symbol for number theoretic groups, in analogy with the Euclidean case; and
prove a basic property of it with regard to $H^{\perp}$. Subsection \ref{sec:lemmas} gives the 
technical lemmas
we need to prove our main Theorem \ref{thm:char} in Subsection \ref{sec:characterization}.
The idea of the proof of Theorem \ref{thm:char}, that distinguishes it essentially and
theoretically from the Euclidean case
in Subsection \ref{sec:Rdtranslates}, is encapsulated at the beginning of 
Subsection \ref{sec:characterization}.

Section \ref{sec:ex} is devoted to fundamental examples that are essential to our point of view.


\subsection{Frames}
\label{sec:frames}

\begin{defn}[Frame]
\label{def:frame}
{\it a.} Let $H$ be a separable Hilbert space over the
field $\FF,$ where $\FF = \RR$ or $\FF = \CC.$ A finite or countably infinite sequence, 
$X = \{x_j\}_{j \in J},$ of elements of $H$ 
is a {\it frame} for $H$ if 
\begin{equation}
\label{eq:framedef}
      \exists A, B > 0  \; \text{such that}  \;
      \forall x \in H, \quad A\norm{x}^2 \leq \sum_{j \in J} |\langle {x},{x_j}\rangle_H|^2 \leq B\norm{x}^2.
\end{equation}
The optimal constants, viz., the supremum over all such $A$ and infimum over all such $B$, are 
called the {\it lower} and {\it upper frame bounds}, respectively. When we refer to {\it frame bounds} 
$A$ and $B$, we shall mean these optimal constants.

{\it b.} A frame $X$ for $H$ is a {\it tight frame} if $A = B.$ If a tight frame has the further property 
that $A = B = 1,$ then the frame is a {\it Parseval frame} for $H.$  

{\it c.}  A frame $X$ for $H$ is {\it equal-norm} 
if each of the elements of $X$ has the same norm. Further, a frame $X$ for $H$ is a {\it unit norm tight frame} 
(UNTF) if each of the elements 
of $X$ has norm $1.$  If $H$ is finite dimensional and $X$ is an UNTF for $H,$ then $X$ is a 
{\it finite unit norm tight frame} ({\it FUNTF}).

{\it d.} A sequence of elements of $H$ satisfying an upper frame bound,
such as $B\norm{x}^2$ in (\ref{eq:framedef}), is a {\it Bessel sequence};
and this second inequality of (\ref{eq:framedef}) is {\it Bessel's inequality}.
\end{defn}

We refer to \cite{daub1992}, \cite{bene1994}, \cite{chri2016} for the theory of frames.


\subsection{Frames of translates for ${\mathbb R}^d$}
\label{sec:Rdtranslates}
$\RR^d$ denotes $d$-dimensional Euclidean space.
Let $f \in L^2({\mathbb R}^d)$, the space of $\CC$-valued square integrable functions taken with Lebesgue 
measure. The {\it Fourier transform} $\widehat{f}$ of  $f \in L^2({\mathbb R}^d)$ is formally defined as
\[
      \widehat{f}(\gamma) = \int_{\RR^d}\,f(x)\,e^{2\pi ix\cdot{\gamma}}\,dx, \quad  \gamma\in \widehat{\RR}^d,
\]
where $ \widehat{\RR}^d= \RR^d$ is considered as the dual locally compact abelian group (LCAG) of
the LCAG $\RR^d$.
Further, formally define the {\it spectral symbol},
\[   
       \Phi(f)(\gamma) = \sum_{m\in \ZZ^d} |\widehat{f}(\gamma + m)|^2, \quad \gamma \in [0,1)^d.
\]
For any fixed $y\in \RR^d$, the translation operator, $\tau_y : L^2({\mathbb R}^d) \longrightarrow L^2({\mathbb R}^d)$, is defined by
$ \tau_{y}(f)(x) = f(x-y)$. For a given $f \in \RR^d$, we consider the space,
\[
           V_f = \overline{\rm span}\{\tau_m f : m\in \ZZ^d\} \subseteq L^2({\mathbb R}^d).
\]  
It is clear that $\Phi(f) \in L^1([0,1)^d)$, and, in fact,  $\norm{\Phi(f)}_{L^1([0,1)^d)} = \norm{f}_{L^2({\mathbb R}^d)}^2$.
Further, it is straightforward to check that if $\{\tau_m f\}$ is 
a Bessel sequence,
then
$\Phi(f) \in L^2([0,1)^d)$.

\begin{thm}[J. Benedetto and Shidong Li, 1992 
\cite{BenLi-1993}, \cite{BenWal1994} Section~3.8]
\label{thm:benli}
Let $f \in L^2({\mathbb R}^d)$. Then, $\{\tau_m f : m\in \ZZ^d\}$ is a frame for $V_f$
if and only if
\[
   \exists\, A, B >0 \; {\rm such}\,{\rm that} \; A \leq \Phi(f) \leq B \;  {\rm on}\, [0,1)^d \setminus N,
\]
where $N = \{\gamma \in [0,1)^d : \Phi(f)(\gamma) = 0 \}$ ($N$ is defined up to sets of measure $0$).
\end{thm}

\begin{rem}[Generalizations]
\label{rem:idea}
The natural generalization of this result to LCAGs $G$, and then some, has been done, e.g.,\cite{CabPat2010},
\cite{GolTou2008}, \cite{HerSikWeiWil2010}, cf. \cite{stro1999}.  Structurally, the generalizations 
typically depend on $\widehat{G}$ having
non-trivial  discrete subgroups that replace $\ZZ^d$, and on thinking 
naturally of $[0,1)^d \subseteq \widehat{\RR}^d$ as a set of 
coset representatives of  $\widehat{\RR}^d/\TT^d$, see Subsection \ref{sec:basictheory}.
This setting does not take into account LCAGs with compact open subgroups, and these include 
many groups that arise in number theory. It is
this setting that we analyze. 
\end{rem}


\section{LCAGs}
\label{sec:lcag}
\subsection{Basic theory}
\label{sec:basictheory}
Let $G$ be a locally compact abelian group (LCAG) with closed subgroup $H$, dual LCAG
$\widehat{G}$, and annihilator subgroup $H^{\perp} \subseteq \widehat{G}$. As a group,
$\widehat{G}$ is defined as the set of all continuous homomophisms, $\gamma : G \longrightarrow
\{ z \in \CC : |z| = 1 \}$, for which $\gamma (x+y) =\gamma(x)\,\gamma(y)$, and where the 
group operation 
on $\widehat{G}$ is defined by $(\gamma_1 + \gamma_2)(x) = \gamma_1 (x)\,\gamma_2(x)$.
It is standard to write $\gamma(x) = (x, \gamma)$, and the continuous homomorphisms
$\gamma$ are called {\it characters} of $G$. The annihilator $H^{\perp} $ of $H$ is defined as
\[
     H^{\perp} = \{ \gamma \in \widehat{G} : \forall x \in H, \; (x, \gamma) = 1 \},
\]
see
\cite{pont1966}, \cite{rudi1962}, \cite{reit1968}, \cite{HewRos1963}, \cite{HewRos1970} for the
basic theory of harmonic analysis on LCAGs beginning with the natural topology on 
$\widehat{G}$. We shall only deal with abelian groups and subgroups. The following properties are basic.

\begin{itemize}

\item $H^{\perp} \subseteq \widehat{G}$ is a closed subgroup.
\item $\widehat{G/H} = H^{\perp}$.
\item $\widehat{G}/H^{\perp} = \widehat{H}$.
\item $H^{\perp}, \,{\rm resp.,} \, H$, is compact $\Longleftrightarrow H, \, {\rm resp.} \, H^{\perp}$, is open.
\item $H \subseteq G$ is an open subgroup $\Longleftrightarrow G/H$ is a discrete group. This assertion only assumes
that $H \subseteq G$ is a subgroup, not necessarily a closed subgroup.
\end{itemize}

The two equalities in the above list designate algebraic and topological isomorphisms. They are proved by
analyzing the natural surjective homomorphisms $h$. For example, for the first equality, begin by
considering the natural surjective homomorphism, $h : G \longrightarrow G/H.$ Then,
the equation, $(x, \gamma) = (h(x), \lambda)$, defines an injection between the elements of $H^{\perp}$
and the continuous characters $\lambda$ defined on $G/H$. It is then straightforward to complete the proof.

A set of {\it coset representatives} of the quotient group $\widehat{G}/H^{\perp}$ 
is denoted by $\Ccal$. $\Ccal$ is defined as a subset of $\widehat{G}$ consisting
of exactly one element of each coset $\Sigma \in \widehat{G}/H^{\perp}$,
and each element $\gamma + H^{\perp} \in \Sigma$ is designated $[\gamma]$
so that generally there are many $\gamma_1, \gamma_2 \in  \widehat{G}$ for which
$[\gamma_1]= [\gamma_2] = \Sigma$. 

Given a set $\Ccal$ of coset representatives of $\widehat{G}/H^{\perp}$.
Based on the fact that the set of all distinct cosets $\Sigma \subseteq \widehat{G}$
is a tiling of $\widehat{G}$, 
we see that each $\gamma \in \widehat{G}$ has a 
{\it unique} representation $\gamma = \sigma + {\eta}_\gamma$, where
$\sigma \in  \Ccal$ and ${\eta}_{\gamma} \in H^{\perp}$.  The {\it cross-section}
mapping,
\[
   \widehat{G}/H^{\perp} \longrightarrow \Ccal, \quad [\gamma] \mapsto [\gamma] \cap \Ccal,
\]
establishes a bijection between $\widehat{G}/H^{\perp}$ and $\Ccal$, that can transmit
the algebraic and topological properties of $\widehat{G}/H^{\perp}$ to $\Ccal$.
Also, because of its use when dealing with fiber bundles, $\Ccal$ is also referred to
as a {\it section}.

Let $\mu = \mu_G$ and $\nu = \nu_{\widehat{G}}$ denote Haar measures on $G$ and $\widehat{G}$,
respectively. $L^1(G)$ is the space of integrable functions on $G$, and the absolutely convergent
{\it Fourier transform} of $f \in L^1(G)$ is the function $\widehat{f}$
defined as
\[
    \widehat{f}(\gamma) = \int_G\,f(x)\,\overline{(x, \gamma)}\,d\mu(x), \, \gamma \in \widehat{G}.
\]
The space of absolutely convergent Fourier transforms is denoted by $A(\widehat{G})$. 
If $f \in L^1(G)$, then 
it can be assumed that the support of $f$ is
$\sigma$-compact; further, the support of each element of
$A(\widehat{G})$ is $\sigma$-compact, see \cite[pages 20--21]{bene1975}. 
The inverse Fourier transform of $F$ defined on $\widehat{G}$ is denoted by $F^{\vee}$.

Let $H \subseteq G$ be a subgroup that is not necessarily closed. However, if $H$ is open, then it is closed.
Further, if $H$ is open and compact, then $H$ and $H^{\perp}$ are compact groups, and the quotients,
$G/H$ and $\widehat{G}/H^{\perp}$, are discrete groups. These facts were noted above. However, in this case,
we can and shall make the following choices of normalization on each of the six interrelated groups
we are discussing, where $H$ is both open and compact:
\begin{itemize}
\item $\mu$ satisfies $\mu(H) = 1$,
\item $\nu$ satisfies $\nu(H^{\perp}) = 1$,
\item $\mu_H =\mu|_H$,
\item $\nu_{H^{\perp}} =\nu|_{H^{\perp}}$,
\item $\mu_{G/H}$ is counting measure, and
\item $\nu_{\widehat{G}/H^{\perp}}$ is counting measure.
\end{itemize}

These choices guarantee that the Fourier transform is an isometry between 
$L^2(G)$ and $L^2(\widehat{G})$, and similarly between $L^2(H)$ and 
$L^2(\widehat{H}) = L^2(\widehat{G}/H^{\perp})$, and between
$L^2(G/H)$ and $L^2(\widehat{G/H})  = L^2(H^{\perp})$, see, e.g.,
\cite{rudi1962}, \cite{reit1968}, and \cite[Section 31.1]{HewRos1970}.
In calculations involving integrals over both $\widehat{G}$
and $H^{\perp}$, we shall use the notation $\gamma \in \widehat{G}$
and $\eta \in H^{\perp}$, and we shall write $d\eta$ instead of $d\nu_{H^{\perp}}(\eta)$.

\begin{rem}[Periodization and Weil's formula]
\label{rem:weil}
Classical Euclidean uniform sampling formulas depend essentially on 
periodization in terms of a discrete subgroup. 
Periodization induces a transformation from the given group to a compact
quotient group, thereby allowing the analysis to be conducted in terms of Fourier series
which lead to sampling formulas. For example, the Shannon wavelet is associated with the 
simplest (and slowly converging) Classical Sampling Formula derived from the {\it sinc}
sampling function. For more general sampling functions $s$, the Classical Sampling Formula
has the form,
\[
     \forall f \in L^2(\mathbb R), \, {\rm for} \, {\rm which} \quad {\rm supp}(\widehat{f}) 
     \subseteq [-\Omega, \Omega], \qquad f =
     T\sum_{n \in \ZZ}\,\widehat{f}(nT)\, \tau_{nT}s,
\]
with convergence in $L^2(\mathbb R)$ norm and uniformly on $\mathbb R$, 
where $0<T\Omega \leq 1$, ${\rm supp}\,(\widehat{s}) \subseteq [-1/T, 1/T]$,
and  $\widehat{s} = 1$ on $[-\Omega, \Omega]$, e.g., see \cite[Chapter 3.10]{bene1997}.
(The notation ${\rm supp}(F)$ designates the support of $F$.)
{\it Weil's formula}
\begin{equation}
\label{eq:weil}
     \int_G\,f(x)\,d\mu(x) =  \int_{G/H}\,\left( \int_H\,f(x+y)\,d\mu_H(y)\right)\,d\mu_{G/H}(x),
\end{equation}
is a far reaching generalization of the idea of periodization, which itself is manifested in the term
$\int_H\,f(x+y)\,d\mu_H(y)$. If two of the three Haar measures in \eqref{eq:weil} are given, then the
third can be normalized so that \eqref{eq:weil} is true on the space $C_c(G)$ of continuous functions 
with compact support. In our setting, with $\mu_H$ as the restriction of $\mu$ to $H$,
the choice of $\mu_{G/H}$ to be counting measure is the appropriate normalization for \eqref{eq:weil}. The 
analogous statement also applies to $\nu, \,\nu_{H^{\perp}}$,  and $\nu_{\widehat{G}/H^{\perp}}$.

A theme of this paper is to
overcome the fact that we do not have the luxury of having non-trivial
discrete subgroups for most of the number-theoretic groups we analyze.
This led to our idea and strategy developed in Sections \ref{sec:transbess} and \ref{sec:main}.

\end{rem}


\subsection{Number theoretic LCAGs -- set-up}
\label{sec:numtheory}
Let $G$ be a LCAG with a compact open subgroup $H$; see,
for example, \cite{kobl1984}, \cite{RamVal1999}, \cite{robe2000}, \cite{serr1979}
for this setting.
The following facts are well known, overlap with some of the assertions in
Subsection \ref{sec:basictheory}, and proofs can be found in these
references as well as those listed in
Subsection \ref{sec:basictheory}. 

\begin{itemize}
    
     \item $H^{\perp}$ is compact open; $G/H$ is discrete; 
     $\widehat{G}/H^{\perp}$ is discrete; $\widehat{G/H} = H^{\perp}$ and is thus compact open.
     
     \item Generally, $G$ and $\widehat{G}$ do not have non-trivial discrete subgroups.
     
\end{itemize}
    
\begin{example}
\label{ex:QpZp}
Given the field $\mathbb Q$ of rational numbers and a prime number $p$. The $p$-adic
absolute value of 
$m/n \in {\mathbb Q} \setminus \{0\}$ is
\[
      \left|\frac{m}{n}\right|_p = p^{-\ord_p(m) + \ord_p(n)},
 \]
where the valuations $\ord_p(m)$ and $\ord_p(n)$ are the exponents of the highest power of 
$p$ dividing $m$ and $n$, respectively. The $p$-adic absolute value $|\cdot|_p$
satisfies the multiplicative equality, $|qr|_p = |q|_p \, |r|_p$,
and the non-Archimedean inequality, $|q+r|_p \leq \max \{|q|_p,\, |r|_p\}$,
for $q,r \in {\mathbb Q}$. The $p$-adic absolute value gives rise to the 
metric $d_p$ defined  as $d_p(q,r) = |q-r|_p$ on ${\mathbb Q}$. As such, $d_p$ satisfies the ultrametric
inequality, $d_p(q,r) \leq \max \{d_p(q,s), \,d_p(s,r)\}$, on ${\mathbb Q}$.

The completion of $\mathbb Q$
with respect to $d_p$ is the complete metric space
$\Qp$ of \emph{$p$-adic numbers},
and the completion of ${\mathbb Z} \subseteq {\mathbb Q}$
with respect to $|\cdot|_p$ is the subspace
$\Zp$ of \emph{$p$-adic integers}. $\Qp$ is a locally compact field of characteristic 
$\charact(\Qp) = 0$. 

As a completion, $\Zp$ is clearly closed, and, in fact, it is the compact unit ball of radius $1$
in $\Qp$. $\Zp$ is also the open ball of radius $1+\epsilon$ in $\Qp$ since the only distances 
actually attained in $\Qp$ are powers of $p$. Another proof that $\Zp$ is open,
and this is also the proof for the analogues of $\Zp$ for the 
more general local fields described in Section \ref{sec:ex}, is by direct calculation
using ultrametric inequalities, that are also induced by absolute values.
Algebraically, $\Zp$ is a compact open subring.
In particular, $\Zp$ is a compact open subgroup of $\Qp$ under addition.

Further, $\Qp$ is separable, and hence second countable since it is a metric space.
As a LCAG under addition, we also 
have that $\Qp = \widehat{\Qp}$ and $\Zp = {\Zp}^{\perp}$.
Moreover, $\Qp/\Zp$ and  
$\widehat{\Qp}/{\Zp}^{\perp}$ are countable discrete groups, for which
the order of each element is a finite power of $p$.
\end{example}
    
\begin{rem}[Countability and $\sigma$-compactness]
\label{rem:countable}
We noted in Subsection \ref{sec:basictheory} that for a LCAG $G$ we can assume without loss of
generality that ${\rm supp}\,f$ is $\sigma$-compact for $f \in L^1(G)$. We just noted in 
Example \ref{ex:QpZp} that $\Qp/\Zp$ and  
$\widehat{\Qp}/{\Zp}^{\perp}$ are countable discrete groups, but the same is not true for
arbitrary number theoretic groups, e.g., see
\cite{bene2003}. For number theoretic groups $G$, we shall be summing over 
the discrete group $G/H$ and a set $\Ccal$ of coset representatives of $\widehat{G}/H^{\perp}$.
However, the sums are well-defined by the aforementioned $\sigma$-compactness,
and so we do not have to assume these sets are countable to prove our assertions,
e.g., see the proof of Proposition \ref{prop:bessel}.

\end{rem}

\section{Translation and Bessel's inequality}
\label{sec:transbess}

\subsection{Translation}
\label{sec:trans}

Our point of view is to think of translation in terms
of a group of operators under
composition as opposed to evaluation on an underlying discrete subgroup.

 Let $G$ be a LCAG, let $H \subseteq G$ be a closed subgroup, and let
 $\Ccal$ be a set of coset representatives of $\widehat{G}/H^{\perp}$.
For any fixed $[x]=x+H \in G/H$, the 
{\it translation operator}, 
\[
      \tau_{[x],\Ccal} : L^2(G) \longrightarrow L^2(G),
\]
is well-defined by the formula,
\[
     \forall\, f\in L^2(G), \quad \tau_{[x],\Ccal}\,f  =  f \ast w_{[x],\Ccal}^{\vee},
\]
where $w_{[x],\Ccal} : \widehat{G} \longrightarrow \CC, \, \gamma \mapsto 
\overline{(x, \eta_\gamma)}$, and $\gamma - \eta_\gamma = \sigma_\gamma \in \Ccal$.
The fact that it is well-defined is a consequence of the validity 
of the Parseval formula in this setting, and because $\widehat{f} \in L^2(\widehat{G})$
and $w_{[x],\Ccal} \in L^{\infty}(\widehat{G})$.
Thus, we think of a group of translation operators  under convolution instead of
an underlying discrete subgroup. Note that $w_{[x],\Ccal}$ depends on $[x]$ and $\Ccal$, 
but not on $x$.

This notion of translation was originally defined for our wavelet theory on local fields \cite{BenBen2004a} (2004).

\begin{example}
\label{ex:transchi}
Given a LCAG $G$ with compact subgroup $H$.
Let $a\in G$ and $\beta\in \widehat{G}$, and let $f(x)=(x,\beta){\mathbb 1}_{a+H}(x)$, i.e.,
\[ 
       \forall  x \in G, \quad f(x) = \begin{cases}
        (x,\beta), & \text{ if } x-a \in H,\\
         0, & \text{ otherwise,} \end{cases} 
\]
where ${\mathbb 1}_X$ is the characteristic function of a set $X$.
Note that $f \in L^1(G)$ since $H$ is a compact set.
We have
\begin{align}
\label{eq:cptH}
\widehat{f}(\gamma) &= \int_{a+H} (x,\beta) \overline{(x,\gamma)} \, dx
= \int_{H} (x+a,\beta-\gamma) \, dx \notag
\\
&= (a,\beta-\gamma)\int_{H} (x,\beta-\gamma) \, dx
= (a,\beta-\gamma) {\mathbb 1}_{\beta+H^{\perp}}(\gamma),
\end{align}
where the last equality follows because the compact set $H$ is  group. In fact, consider the cases
$\gamma \in \beta + H^{\perp}$ and $\gamma \notin \beta + H^{\perp}$. If $\gamma \in \beta + H^{\perp}$
and $y \in H$, then $(y, \beta - \gamma) = \overline{(y, \gamma - \beta)} =1$; and so the last integral
in equation (\ref{eq:cptH}) is $1 = {\mathbb 1}_{\beta+H^{\perp}}(\gamma)$ for $\gamma \in \beta + H^{\perp}$.
If $\gamma \notin \beta + H^{\perp}$, then there is $y \in H$ for which $\overline{(y,\gamma - \beta)} \neq 1$,
and we compute
\[
     \int_H\, (x,\gamma - \beta)\,dx = \overline{(y,\gamma - \beta)} \,\int_H\, (x,\gamma - \beta)\,dx;
\]
from which we can conclude that $ \int_H\, (x,\gamma - \beta)\,dx = 0$, which can be written as  
${\mathbb 1}_{\beta+H^{\perp}}(\gamma)$ for such $\gamma \notin \beta + H^{\perp}$.

Hence, for any $[b]\in G/H$, we compute
\[ \widehat{\tau_{[b],\Ccal}f}(\gamma)
= (a,\beta-\gamma) \overline{(b,\eta_\gamma)} {\mathbb 1}_{\beta+H^{\perp}}(\gamma)
= (a,\beta-\gamma) \overline{(b,\gamma - \sigma_\beta)} {\mathbb 1}_{\beta+H^{\perp}}(\gamma), 
\]
since, for all $\gamma\in \beta+H^{\perp}$, we have
$\sigma_\gamma=\sigma_\beta$, whence $\eta_\gamma = \gamma-\sigma_\beta$.
Therefore,
\begin{align*}
\tau_{[b],\Ccal}f(x) &=
\int_{\beta+H^{\perp}} (a,\beta-\gamma) (b,\sigma_\beta - \gamma) (x,\gamma) \, d\gamma
= \int_{H^{\perp}} \overline{(a,\gamma)} \overline{(b,\beta+\gamma-\sigma_{\beta})}
(x,\beta+\gamma) \, d\gamma
\\
&= (b,\sigma_{\beta})(x-b,\beta) \int_{H^{\perp}} (x-a-b,\gamma) \, d\gamma
= (b,\sigma_\beta) (x-b,\beta) {\mathbb 1}_{a+b+H}(x)
\\
&= (b,\sigma_\beta) f(x-b).
\end{align*}
\end{example}

Example~\ref{ex:transchi} illustrates that
the translation operators $\tau_{[b],\Ccal}$ are related to ordinary translation,
but are not quite the same. The key advantage they provide is that they form
a group isomorphic to $G/H$, even though $G$ generally does not contain
a subgroup isomorphic to $G/H$. 
Indeed, for $[x],[y]\in G/H$, it is easy to check that
$\tau_{[x],\Ccal} \circ \tau_{[y],\Ccal} = \tau_{[x]+[y],\Ccal}$,
with $\tau_{[x],\Ccal}= \tau_{[y],\Ccal}$ if and only if $[x]=[y]$, i.e., if and only if $x+H=y+H$,
see \cite[Remark~2.3]{BenBen2004a}.


\subsection{$V_{\Ccal,f}$ and $\Phi_{\Ccal,f}(g)$}
\label{sec:spanspecsym}

Let $G$ be a LCAG, let $H \subseteq G$ be a compact open subgroup, and let
 $\Ccal$ be a set of coset representatives of $\widehat{G}/H^{\perp}$.
Take $f \in L^2(G)$ and define the {\it closed span of translates},
\[
    V_{\Ccal,f} = \overline{\rm span}\,\{\tau_{[x],\Ccal}\,f : [x] \in G/H\},
\]
and the {\it spectral symbol},
\begin{equation}
\label{eq:phifg}
 \forall\, g \in L^2(G), \quad \Phi_{\Ccal,f}(g)(\eta) = 
 \sum_{\sigma \in \Ccal}\,\widehat{g}(\eta + \sigma)\overline{\widehat{f}(\eta + \sigma)}, \; \eta \in H^{\perp}.
\end{equation}
Clearly, $\Phi_{\Ccal,f}(g) \in L^1(H^{\perp})$. Denote $ \Phi_{\Ccal,f}(f)$ as   $\Phi_{\Ccal}(f)$.

\begin{rem}[Orthonormal basis of characters]
\label{rem:ortho}
The discrete set $G/H$ of characters of the compact group $H^{\perp}$ is 
easily seen to be orthonormal. Also, the set $P$ of trigonometric polynomials,
\[
   \Theta_F(\eta) = \sum_{[x]\in F}\,c_{[x]}\, \overline{([x], \eta)}, \quad \eta \in H^{\perp}, 
\]
where $F \subseteq G/H, {\rm card}(F) < \infty$, and $c_{[x]} \in \CC$, is a dense sub-algebra 
of $C(H^{\perp})$, the function algebra of continuous functions on $H^{\perp}$ taken with the 
sup-norm. This is a consequence of the Stone-Weierstrass theorem since $P$ is closed under
complex conjugation and $G/H$ separates points on $H^{\perp}$. By this density, $P$ is also dense in
$L^2(H^{\perp})$. By the aforementioned orthonormality, and the fact that
$P$ is dense in $L^2(H^{\perp})$, a standard Hilbert space argument shows that $G/H$ is an
orthonormal basis for $L^2(H^{\perp})$, see, e.g., \cite[pages 27--28]{GohGol1981}. We shall use the 
$L^2(H^{\perp})$-norm convergence of Fourier series theorem in part {\it iii}
of the proof of Proposition \ref{prop:bessel} and in part {\it iv} of the proof of Lemma \ref{lem:54}.
\end{rem}

Proposition \ref{prop:bessel} is necessary for the proof of our main result, Theorem \ref{thm:char}.
The Euclidean analogue of Proposition \ref{prop:bessel} is 
required for the proof of Theorem \ref{thm:benli}. Its proof may appear
more direct than
what follows for number theoretic groups, but the idea is the same. In Proposition \ref{prop:bessel}
we are integrating
over the compact group $H^{\perp}$ instead of $\RR^d/\ZZ^d$, which is really the section
$[0, 1)^d$, and then summing over the section $\Ccal$ instead of the discrete subgroup $\ZZ^d$.

 \begin{prop}
 \label{prop:bessel}
 Let $G$ be a LCAG, let $H$ be a compact open subgroup,
 and let
 $\Ccal$ be a set of coset representatives of $\widehat{G}/H^{\perp}$. 
 Let $f \in L^2(G)$, and assume the sequence, $\{\tau_{[x],\Ccal}\,f : [x] \in G/H \}$,
 satisfies Bessel's inequality,
 \begin{equation}
 \label{eq:bessel}
    \exists\,B>0 \; {\rm such}\,{\rm that} \; \forall\,g \in V_{\Ccal,f}, \quad 
    \sum_{[x]\in G/H}\, |\langle g, \tau_{[x],\Ccal}\,f \rangle_{L^2(G)}|^2 \leq B\, {\norm{g}}_{L^2(G)}.
 \end{equation}
 Then, $\Phi_{\Ccal,f}(g) \in L^2(H^{\perp})$ for each $g \in L^2(G)$, and,
 in particular, $\Phi_{\Ccal}(f) \in L^2(H^{\perp})$.
 \end{prop}

 \begin{proof}
   {\it i.} Let $g \in L^2(G)$. As observed after \eqref{eq:phifg}, we have $\Phi_{\Ccal,f}(g) \in L^1(H^{\perp})$.
 Then,
 \begin{equation}
 \label{eq:ghathperp}
      \langle g, f \rangle_{ L^2(G)} =  \int_{H^{\perp}}\,\Phi_{\Ccal,f}(g)(\eta)\,d\eta.
 \end{equation}
 In fact, we compute
\begin{align*}
\int_{H^{\perp}}\,\Phi_{\Ccal,f}(g)(\eta)\,d\eta
&= \sum_{\sigma \in \Ccal}\, \int_{H^{\perp}}\,\widehat{g}(\eta + \sigma)\,
\overline{\widehat{f}(\eta + \sigma)} \,d\eta
\\
&= \sum_{\sigma \in \Ccal}\,\int_{\sigma + H^{\perp}}\widehat{g}(\gamma)\,
\overline{\widehat{f}(\gamma)}\,d\gamma
=  \langle \widehat{g}, \widehat{f} \rangle _{ L^2(\widehat{G})} = \langle g, f \rangle _{ L^2(G)},
\end{align*}
by means of the Fubini-Tonelli theorem \cite[pages 132--140]{BenCza2009}, a change of variables,
a partitioning of $\widehat{G}$, and the Plancherel theorem, see this rationale in reverse
to justify the calculation after \eqref{eq:besselhyp}.
 
Further, we can write
\begin{equation}
\label{eq:transmod}
        \widehat{(\tau_{[x],\Ccal}f)}(\eta + \sigma) 
       =  (\widehat{f}\,w_{[x],\Ccal})(\eta + \sigma)  
       = \overline{(x, \eta)}\,\widehat{f}(\eta + \sigma),
\end{equation}
where we have used the definition of translation in Subsection \ref{sec:trans}
and the unique representation for coset representatives for the last equality.

{\it ii.} Note that
\begin{equation*}
        \Phi_{\Ccal,f}(g) \in L^2(H^{\perp}) \; \Longleftrightarrow \;
        \int_{H^{\perp}} \, \left|\sum_{\sigma \in \Ccal}\, \widehat{g}(\eta + \sigma)\,
     \overline{\widehat{f}(\eta + \sigma)}\right|^2 \,d\eta < \infty.
\end{equation*}
Taking into account Remark \ref{rem:countable}, we know 
by assumption \eqref{eq:bessel} that
\[
      \sum_{[x]\in G/H}\, |\langle g, \tau_{[x],\Ccal}\,f \rangle_{ L^2(G)}|^2 < \infty;
\]
and so,
\begin{equation}
\label{eq:besselhyp}
        \sum_{[x] \in G/H}\, \left| \int_{H^{\perp}} \,(x, \eta)\, \left( \sum_{\sigma \in \Ccal}\, \widehat{g}(\eta + \sigma)
     \overline{\widehat{f}(\eta + \sigma)} \right) \, d\eta\right|^2 
         =  \sum_{[x]\in G/H}\, |\langle g, \tau_{[x],\Ccal}\,f \rangle_{L^2(G)}|^2  
     < \infty.
\end{equation}
In fact, repeating the calculation after \eqref{eq:ghathperp}, but in reverse order, we have
\begin{align*}
    \langle g, \tau_{[x],\Ccal} & \,f \rangle_{ L^2(G)}
    = \int_{\widehat{G}}\,\widehat{g}(\gamma)\,
    \overline{ \widehat{\tau_{[x],\Ccal}\,f }(\gamma)}\,d\gamma
\\
    & = \sum_{\sigma \in \Ccal}\,\int_{\sigma + H^{\perp}}\,\widehat{g}(\gamma)\,
    \overline{ \widehat{\tau_{[x],\Ccal}\,f }(\gamma)}\,d\gamma
    = \sum_{\sigma \in \Ccal}\,\int_{H^{\perp}}\,\widehat{g}(\eta + \sigma)\,
    \overline{ \widehat{\tau_{[x],\Ccal}\,f } (\eta + \sigma)}\,d\eta
\\
    & = \sum_{\sigma \in \Ccal}\,\int_{H^{\perp}}\,\widehat{g}(\eta + \sigma)\,
    \overline{\overline{(x, \eta) }\,\widehat{f}(\eta + \sigma)}\,d\eta
    =  \int_{H^{\perp}} \,(x, \eta)\, \left( \sum_{\sigma \in \Ccal}\, \widehat{g}(\eta + \sigma)
     \overline{\widehat{f}(\eta + \sigma)} \right) \, d\eta,
\end{align*}
where the first equality follows from the Plancherel theorem, the second by partitioning $\widehat{G}$,
the third by change of variables and carrying out the calculation on the group $H^{\perp}$
instead of $\widehat{G}$, the fourth by the translation-modulation property \eqref{eq:transmod},
and the last by the Fubini-Tonelli theorem.

Therefore, since
\[
         \Phi_{\Ccal,f}(g)(\eta) = \sum_{\sigma \in \Ccal}\, \widehat{g}(\eta + \sigma)
     \overline{\widehat{f}(\eta + \sigma)} \in L^1(H^{\perp}),
\]
we see from \eqref{eq:besselhyp}
that the sequence, $\{c_{[x]}(g,f)  : [x] \in G/H\}$,
of Fourier coefficients,
\begin{equation}
\label{eq:coeff}
            c_{[x]}(g,f) =  \int_{H^{\perp}}\,\Phi_{\Ccal,f}(g)(\eta)\,(x, \eta)\, d\eta,
\end{equation}
is an element of $\ell^2(G/H)$. 
    
 {\it iii.}   
The set $\{ ([u], \cdot) : [u] \in G/H \}$ of characters  
of $H^{\perp}$ is an orthonormal sequence in $L^2(H^{\perp})$,
see Remark \ref{rem:ortho}. 
Since  $\{c_{[x]}(g,f) \} \in \ell^2(G/H)$, an elementary Hilbert space 
Cauchy sequence argument shows that the Fourier series $F_f(g)$, defined as
\[
        F_f(g)(\eta) = \sum_{[u] \in G/H}\, c_{[u]}(g,f)\,\overline{([u], \eta)}, \quad \eta \in H^{\perp},
\]
is a well-defined element of $L^2(H^{\perp})$, where convergence is in $L^2(H^{\perp})$-norm.
Further, by the orthonormality, and another standard Hilbert space calculation,
\begin{equation}
\label{eq:Ffg}
      c_{[u]}(g,f) =  \big\langle  F_f(g)(\cdot), \overline{([u], \cdot)}\big\rangle_{L^2(H^{\perp})} =
      \int_{H^{\perp}}\,F_f(g)(\eta)\,([u], \eta)\,d\eta,
\end{equation}
where $\langle \cdot  , \cdot \rangle_{L^2(H^{\perp})}$ is the inner product on the Hilbert space $L^2(H^{\perp})$.

Combining equations \eqref{eq:coeff} and \eqref{eq:Ffg}, and noting that 
$\Phi_{\Ccal,f}(g), F_f(g) \in L^1(H^{\perp})$, we have
\[
    \forall\,[u] \in G/H, \quad \int_{H^{\perp}}\,\big(F_f(g)(\eta) - \Phi_{\Ccal,f}(g)(\eta) \big)\,\overline{([u], \eta)}\,d\eta
    = 0.
\]
Consequently, by the $L^1$-uniqueness theorem for Fourier series, we have $\Phi_{\Ccal,f}(g) = F_f(g) \; {\rm a.e.}$,
and so we can conclude that $\Phi_{\Ccal,f}(g) \in L^2(H^{\perp})$ since $F_f(g) \in L^2(H^{\perp})$.
\end{proof}


\section{Main result}
\label{sec:main}

\subsection{Lemmas}
\label{sec:lemmas}

Throughout this section $G$ is a LCAG with compact open subgroup $H$
and 
 $\Ccal$ is a set of coset representatives of $\widehat{G}/H^{\perp}$.

\begin{lem}
\label{lem:53}
Let $f \in  L^2(G)$ and define $g_F = \sum_{[x]\in F}\,c_{[x]} \tau_{[x],\Ccal}\,f$ and
\[
    \Theta_F(\eta) = \sum_{[x]\in F}\,c_{[x]}\, \overline{([x], \eta)}, \quad \eta \in H^{\perp},
\]
where $F \subseteq G/H, {\rm card}(F) < \infty$, and $c_{[x]} = c_{[x]}(g_F) \in \CC$.
Then, $g_F \in L^2(G), \Theta_F \in L^{\infty}(H^{\perp})$, and
\begin{equation}
\label{eq:21}
       \norm{g_F}_{L^2(G)}^2 = \int_{H^{\perp}}\,|\Theta_F(\eta)|^2\,\Phi_{\Ccal}(f)(\eta)\,d\eta
       < \infty.
    \end{equation}
\end{lem}
\begin{proof}
For each $\gamma \in \widehat{G}$, we designate its unique representation in terms of the given
section $\Ccal$ by $\gamma = \sigma_{\gamma} + \eta_{\gamma}$, where $\sigma_{\gamma} \in \Ccal$
and $\eta_{\gamma} \in H^{\perp}$. Using the fact that 
\[
\widehat{(\tau_{[x],\Ccal}\, f)}(\gamma) = 
\widehat{f}(\gamma)\,w_{[x],\Ccal}(\gamma) = \widehat{f}(\gamma) \, \overline{([x], \eta_{\gamma})},
\]
we make the following computation which invokes the Fubini-Tonelli theorem:
\begin{align*}
     \norm{g_F}_{L^2(G)}^2
     &= \int_{\widehat{G}}\,\widehat{g_F}(\gamma)\overline{\widehat{g_F}(\gamma)}\,d\gamma
     =  \sum_{[x], [y] \in F}\,c_{[x]}\overline{c_{[y]}}\,\int_{\widehat{G}}
     \widehat{f}(\gamma)\overline{\widehat{f}(\gamma)}
     \overline{([x], \eta_{\gamma})}([y], \eta_{\gamma})\,d\gamma
\\
     &= \sum_{[x], [y] \in F}\,c_{[x]}\overline{c_{[y]}} \sum_{\sigma \in \Ccal}\int_{H^{\perp}}
     \widehat{f}(\eta + \sigma)\overline{\widehat{f}(\eta + \sigma)}
     \overline{([x], \eta)}([y], \eta)\,d\eta
\\
     &=  \sum_{[x], [y] \in F}\,c_{[x]}\overline{c_{[y]}} \sum_{\sigma \in \Ccal}\int_{H^{\perp}}
     \big|\widehat{f}(\eta + \sigma)\big|^2\,\overline{([x], \eta)}([y], \eta)\,d\eta
\\
     &= \int_{H^{\perp}} \Bigg(\sum_{[x] \in F}c_{[x]}\,\overline{([x],\eta)}\Bigg)
             \overline{\Bigg(\sum_{[y]\in F} c_{[y]}\overline{([y],\eta)}\Bigg)}\,
             \left(\sum_{\sigma \in \Ccal}\big|\widehat{f}(\eta + \sigma)\big|^2 \right)\,d\eta
\\
     &= \int_{H^{\perp}}\,|\Theta_F(\eta)|^2\,\Phi_{\Ccal}(f)(\eta)\,d\eta.
\end{align*}
Here, $d\gamma$ represents Haar measure on $\widehat{G}$, thinking of the sequence 
$\{\sigma + H^{\perp} : \sigma \in \Ccal \}$ of cosets
of $\widehat{G}$ as a partition of $\widehat{G}$, 
where $d\eta$ represents normalized Haar measure on $H^{\perp}$ so that the last equality is also
valid with $d\eta$ replaced by $d\gamma$, and where $\eta$ in the third equality 
 corresponds to $\eta_{\gamma}$
in the representation $\gamma = \sigma + \eta_{\gamma}$.
\end{proof}

The following is proved by a routine argument in real analysis.

\begin{lem}
\label{lem:52}
Let $f \in  L^2(G)$, and consider the frame condition,
\begin{equation}
\label{eq:tranframe}
      A\,\norm{g}_{L^2(G)}^2 \leq  \sum_{[x]\in G/H}\, |\langle g, \tau_{[x],\Ccal}\,f \rangle_{L^2(G)}|^2 \leq
      B\,\norm{g}_{L^2(G)}^2.
\end{equation}
Then, \eqref{eq:tranframe} is valid for each $g \in V_{\Ccal,f}$ if and only if 
\eqref{eq:tranframe} is valid for each $g \in {\rm span}\{\tau_{[x],\Ccal}\,f: [x] \in G/H \}$.
\end{lem}

\begin{lem}
\label{lem:54}
Let $f \in  L^2(G)$ and define $g_F = \sum_{[x]\in F}\,c_{[x]} \tau_{[x],\Ccal}\,f$ and
\[
    \Theta_F(\eta) = \sum_{[x]\in F}\,c_{[x]}\, \overline{([x], \eta)}, \quad \eta \in H^{\perp},
\]
where $F \subseteq G/H, {\rm card}(F) < \infty$, and $c_{[x]} = c_{[x]}(g_F) \in \CC$.
Assume $\Phi_{\Ccal}(f) \in L^2(H^{\perp})$. Then, $g_F \in L^2(G), \Theta_F \in L^{\infty}(H^{\perp})$, and
\begin{equation}
\label{eq:22}
         \sum_{[x]\in G/H}\, |\langle g_F, \tau_{[x],\Ccal}\,f \rangle_{L^2(G)}|^2
         = \int_{H^{\perp}}\,|\Theta_F(\eta)|^2\,\Phi_{\Ccal}(f)(\eta)^2\,d\eta < \infty.
\end{equation}
\end{lem}

\begin{proof}
{\it i.} We first estimate the $L^2$-norm of $g_F$,
using Minkowski's inequality and then the Plancherel formula,
taking into account that $w_{[x],\Ccal} $ is unimodular:

\begin{align*}
      \norm{g_F}_{L^2(G)}
      & =
      \Bigg( \int_G\,\bigg( \sum_{[x]\in F}\, c_{[x]}\, \tau_{[x],\Ccal}\,f(y) \bigg)\,
      \bigg( \sum_{[x]\in F}\, \overline{c_{[x]}}\, \overline{\tau_{[x],\Ccal}\,f(y)}  \bigg)\,dy  \Bigg)^{\frac{1}{2}}
      \\
      & \leq \sum_{[x]\in F}\,\bigg( \int_G\,  \big|c_{[x]}\, \tau_{[x],\Ccal}\,f(y)\big|^2 \, dy  \bigg)^{\frac{1}{2}}
      \\
      & = 
      \sum_{[x]\in F}\, \big|c_{[x]}\big| \,\bigg( \int_G\,  \big|  f \ast w_{[x],\Ccal}^{\vee}(y)   \big|^2 \, dy  \bigg)^{\frac{1}{2}}
      \\
      & =
       \sum_{[x]\in F}\, \big|c_{[x]}\big| \,\bigg( \int_{\widehat{G}}\, 
       \big|  \widehat{f}(\gamma)\,w_{[x],\Ccal}(\gamma)   \big|^2 \, d\gamma  \bigg)^{\frac{1}{2}}
      \\
      & =  
      \norm{f}_{L^2(G)}\,\Big( \sum_{[x]\in F}\,\big|c_{[x]}\big|\Big) < \infty.
\end{align*}
In particular, $g_F  \in L^2(G)$, as we knew from its definition and Lemma \ref{lem:53}.

{\it ii.} Let $K \subseteq G/H$ satisfy ${\rm card}\,(K) < \infty$.  By the Parseval-Plancherel formula
we compute 
\begin{align*}
   0 \leq \sum_{[x]\in K \subseteq G/H}\, |\langle g_F, \tau_{[x],\Ccal}\,f \rangle_{L^2(G)}|^2
   & =  \sum_{[x]\in K}\, \langle g_F, \tau_{[x],\Ccal}\,f \rangle_{L^2(G)} \, \overline{\langle g_F, \tau_{[x],\Ccal}\,f \rangle_{L^2(G)}}
   \\
   & = \sum_{[x]\in K}\,\int_G\, g_F(y)\,\overline{f \ast w_{[x],\Ccal}^{\vee}(y)} dy\, \int_G\, 
   \overline{g_F(y)\, \overline{f \ast w_{[x],\Ccal}^{\vee}(y)}  }\,dy
   \\
   & = \sum_{[x]\in K}\, \int_{\widehat{G}} \,\widehat{g_F}(\gamma)\, \overline{\widehat{f}(\gamma)\,\overline{(x, \eta_{\gamma})}}\,d{\gamma}\,
   \overline{\int_{\widehat{G}}\,\widehat{g_F}(\gamma)  \,\overline{\widehat{f}(\gamma)\,\overline{(x, \,\eta_{\gamma}} )}d{\gamma}} 
\end{align*}
\begin{equation}
\label{eq:cosrep}
=
   \sum_{[y],[z]\in F}\,c_{[y]}\,\overline{c_{[z]}}\,\int_{\widehat{G}}\,\big| \widehat{f}(\gamma)\big|^2\,\overline{(y,\eta_{\gamma})}   
   \, \Bigg(\int_{\widehat{G}} \big|\widehat{f}(\lambda)\big|^2\,(z,\eta_{\lambda}) \bigg(\sum_{[x]\in K}\,(x, \eta_{\gamma} -\eta_{\lambda})\bigg)  
   \,d\lambda \Bigg)\,d\gamma,
\end{equation}
where
\[
     \widehat{g_F}(\gamma) = \sum_{[y]\in F}\,c_{[y]}\,\overline{(y, \eta_{\gamma})}\, \widehat{f}(\gamma), \quad \gamma \in \widehat{G},
\]
and $\gamma - \eta_{\gamma} = \sigma_{\gamma} \in \Ccal$ is the unique representation of each  $\gamma \in \widehat{G}$.

We now rewrite \eqref{eq:cosrep} in terms of integration over $H^{\perp}$. To this end, we begin by computing
the inner integral of \eqref{eq:cosrep} as follows using the fact that each $\lambda \in \widehat{G}$ has the unique
representation $\lambda = \sigma + \eta_{\lambda}$ for some $\sigma \in \Ccal$ and some $\eta_{\lambda} \in H^{\perp}$.
\begin{align*}
  \int_{\widehat{G}} \big|\widehat{f}(\lambda)\big|^2\,(z,\eta_{\lambda}) \bigg(\sum_{[x]\in K}\,(x, \eta_{\gamma} -\eta_{\lambda})\bigg)  
   \,d\lambda 
   & =
   \sum_{\sigma \in \Ccal}\,\int_{\sigma + H^{\perp}}\,
   \big|\widehat{f}(\lambda)\big|^2\,(z,\eta_{\lambda}) \bigg(\sum_{[x]\in K}\,(x, \eta_{\gamma} -\eta_{\lambda})\bigg) \,d\lambda
   \\
   & = 
   \sum_{\sigma \in \Ccal}\,\int_{H^{\perp}}\,
   \big|\widehat{f}(\sigma + \eta)\big|^2\,(z,\eta) \bigg(\sum_{[x]\in K}\,(x, \eta_{\gamma} -\eta)\bigg) \,d\eta
   \\
   & =
   \int_{H^{\perp}}\, \sum_{\sigma \in \Ccal}\,
   \big|\widehat{f}(\sigma + \eta)\big|^2\,(z,\eta) \bigg(\sum_{[x]\in K}\,(x, \eta_{\gamma} -\eta)\bigg) \,d{\eta}, 
\end{align*}
where $d{\lambda}$ denotes Haar measure on $\widehat{G}$ and 
$d{\eta}$ denotes Haar measure on $H^{\perp}$. The last equality also follows
from the Fubini-Tonelli theorem, actually Tonelli's theorem \cite[page 138]{BenCza2009} 
in this case, since
$(z,\eta) \big(\sum_{[x]\in K}\,(x, \eta_{\gamma} -\eta)\big) \in L^{\infty}(H^{\perp})$
and $\Phi_{\Ccal}(f) \in L^1(H^{\perp})$, see Subsection \ref{sec:spanspecsym}.
Thus, we have
\begin{equation}
\label{eq:fourser1}
\int_{\widehat{G}} \big|\widehat{f}(\lambda)\big|^2\,(z,\eta_{\lambda}) \bigg(\sum_{[x]\in K}\,(x, \eta_{\gamma} -\eta_{\lambda})\bigg)  
   \,d\lambda
    =  \int_{H^{\perp}}\, \Phi_{\Ccal}(f)(\eta)
   \,(z,\eta) \bigg(\sum_{[x]\in K}\,(x, \eta_{\gamma} -\eta)\bigg) \,d{\eta},
\end{equation}
and, with a similar calculation for its outer integral, \eqref{eq:cosrep} becomes
\[
 \sum_{[y],[z]\in F}\,c_{[y]}\,\overline{c_{[z]}}\,\Bigg( \sum_{\sigma \in \Ccal}\, \int_{\sigma + H^{\perp}}\,
    \bigg( \big| \widehat{f}(\gamma)\big|^2\,\overline{(y,\eta_{\gamma})} 
    \Big(  \int_{H^{\perp}}\, \Phi_{\Ccal}(f)(\eta)
   \,(z,\eta) \Big(\sum_{[x]\in K}\,(x, \eta_{\gamma} -\eta)\Big) \,d{\eta} \Big) \bigg)\,d\gamma \Bigg) \\
\]
\[
   =
 \sum_{[y],[z]\in F}\,c_{[y]}\,\overline{c_{[z]}}\,\Bigg( \sum_{\sigma \in \Ccal}\, \int_{H^{\perp}}\,
     \bigg(\big| \widehat{f}(\sigma + \nu)\big|^2\,\overline{(y,\nu)} 
     \Big( \int_{H^{\perp}}\, \Phi_{\Ccal}(f)(\eta)
   \,(z,\eta) \Big(\sum_{[x]\in K}\,(x, \nu -\eta)\Big) \,d{\eta} \Big)\bigg)\,d\nu \Bigg)
     \\
 \]
 \[
 =  
      \sum_{[y],[z]\in F}\,c_{[y]}\,\overline{c_{[z]}}\, \int_{H^{\perp}}\,\Bigg(\sum_{\sigma \in \Ccal}\, 
     \big| \widehat{f}(\sigma + \nu) \big|^2\,\overline{(y,\nu)} \,
     \bigg(\int_{H^{\perp}}\, \Phi_{\Ccal}(f)(\eta)
   \,(z,\eta) \Big(\sum_{[x]\in K}\,(x, \nu -\eta)\Big) \,d{\eta} \bigg)\Bigg)\,d\nu
     \\
\]
\begin{equation}
\label{eq:phiphi}
=  
      \sum_{[y],[z]\in F}\,c_{[y]}\,\overline{c_{[z]}}\, \int_{H^{\perp}}\,\Bigg(\Phi_{\Ccal}(f)(\nu)
     \,\overline{(y,\nu)} \,
     \bigg(\!\!\int_{H^{\perp}}\, \Phi_{\Ccal}(f)(\eta)
   \,(z,\eta) \Big(\sum_{[x]\in K}\,(x, \nu -\eta)\Big) \,d{\eta} \bigg)\Bigg)\,d\nu,
\end{equation}
where $d{\gamma}$ denotes Haar measure on $\widehat{G}$, and 
$d{\eta}$ and $d{\nu}$ denote Haar measure on $H^{\perp}$.

{\it iii.} Now let us consider the right side of the inner integral \eqref{eq:fourser1}
in terms of Fourier series. After the previous calculation, where Haar measure
$d\nu$ was introduced, this right side is
\begin{align*}
   \int_{H^{\perp}}\, \Phi_{\Ccal}(f)(\eta)
   \,(z,\eta) \Bigg(\sum_{[x]\in K}\,(x, \nu -\eta) \Bigg) \,d{\eta}
    &= \sum_{[x]\in K}\, \bigg(\int_{H^{\perp}}\, \Phi_{\Ccal}(f)(\eta)
   \,(z - x,\eta) \, d{\eta}\bigg)\,(x, \nu)
   \\
   &= \sum_{[x]\in K}\, \bigg(\int_{H^{\perp}}\, \Big( \Phi_{\Ccal}(f)(\eta)
   \,(z,\eta) \Big) \, \overline{(x,\eta)} d{\eta}\bigg)\,(x, \nu),
\end{align*}
which is precisely 
the $K$-th partial Fourier series sum,
\[
      S_K \big( \Phi_{\Ccal}(f)(\cdot)\,(z,\cdot) \big)(\nu),
\]
of the function $\Phi_{\Ccal}(f)(\cdot)\,(z,\cdot) \in L^1(H^{\perp})$
with Fourier coefficients, 
\[
     \int_{H^{\perp}}\, \Big( \Phi_{\Ccal}(f)(\eta)
   \,(z,\eta) \Big) \, \overline{(x,\eta)} d{\eta}, \quad [x] \in K \subseteq G/H.
\]
   
Combining this calculation and notation with those of parts {\it i} and {\it ii},
we have shown that
\begin{equation}
\label{eq:fourser2}
 0 \leq  \sum_{[x]\in K \subseteq G/H}\, |\langle g_F, \tau_{[x],\Ccal}\,f \rangle_{L^2(G)}|^2
  = \sum_{[y],[z]\in F}\,c_{[y]}\,\overline{c_{[z]}}\, \int_{H^{\perp}}\,\Phi_{\Ccal}(f)(\nu)
     \,\overline{(y,\nu)} \,S_K \big( \Phi_{\Ccal}(f)(\cdot)\,(z,\cdot) \big)(\nu) \, d\nu.
\end{equation}

{\it iv.} We shall now prove that 
\begin{equation}
\label{eq:fourser3}
  {\rm lim}_K \,\int_{H^{\perp}} \Phi_{\Ccal}(f)(\nu)
     \,\overline{(y,\nu)} \,S_K \big( \Phi_{\Ccal}(f)(\cdot)\,(z,\cdot) \big)(\nu) \, d\nu
     =  \int_{H^{\perp}}\,\Phi_{\Ccal}(f)^{2}(\nu)\, (z-y,\nu)\,d\nu,
\end{equation}
where ${\rm lim}_K$ denotes the limit as the finite subsets $K \subseteq G/H$
increase to all of the discrete group $G/H$, see Remark \ref{rem:countable} that obviates
countability concerns. In fact, the right side of \eqref{eq:fourser3} exists since 
$\Phi_{\Ccal}(f) \in L^2(H^{\perp})$, and we have the inequality,
\begin{align*}
   \bigg |\int_{H^{\perp}}\,\Phi_{\Ccal}(f)(\nu) &
     \,\overline{(y,\nu)} \,\Big(S_K \big( \Phi_{\Ccal}(f)(\cdot)\,(z,\cdot) \big)(\nu) 
     -  \Phi_{\Ccal}(f)(\nu)\,(z,\nu) \Big)\, d\nu  \bigg|
\\
& \leq \norm{\Phi_{\Ccal}(f)}_{L^2(H^{\perp})}\,
        \big\| S_K \big( \Phi_{\Ccal}(f)(\cdot)\,(z,\cdot) \big)(\nu) 
     -  \Phi_{\Ccal}(f)(\nu)\,(z,\nu)\big\|_{L^2(H^{\perp})}.
\end{align*}
The right side of this inequality tends to $0$ as $K$ increases to $G/H$ by
the elementary $L^2(H^{\perp})$-norm convergence of Fourier series
on the compact abelian group $H^{\perp}$.

{\it v.} Clearly, with non-negative terms in the sum, the limit,
\[
    {\rm lim}_K \,\sum_{[x]\in K \subseteq G/H}\, |\langle g_F, \tau_{[x],\Ccal}\,f \rangle_{L^2(G)}|^2,
\]
 exists and equals the left side of \eqref{eq:22}. Consequently, \eqref{eq:fourser2} and 
 \eqref{eq:fourser3} combine to give
\begin{align*}
    \sum_{[x]\in \subseteq G/H}\, |\langle g_F, \tau_{[x],\Ccal}\,f \rangle_{L^2(G)}|^2
     & =
     \sum_{[y],[z]\in F}\,c_{[y]}\,\overline{c_{[z]}}\,\int_{H^{\perp}}\,\Phi_{\Ccal}(f)^{2}(\nu)\, (z-y,\nu)\,d\nu
     \\
     & =
      \int_{H^{\perp}}\,|\Theta_F(\gamma)|^2\,\Phi_{\Ccal}(f)(\gamma)^2\,d\gamma < \infty,
\end{align*}
and this is \eqref{eq:22}, the desired result.
\end{proof}

\begin{lem}
\label{lem:55}
Let $f \in  L^2(G)$ and assume $\Phi_{\Ccal}(f) \in L^2(H^{\perp})$.
The sequence, $\{\tau_{[x],\Ccal}\,f : [x]\in G/H \}$, is a frame for $V_{\Ccal,f}$
with frame bounds $A$ and $B$ if and only if there are positive constants
$A$ and $B$ such that for all trigonometric polynomials,
\[
    \Theta_F(\eta) = \sum_{[x]\in F}\,c_{[x]}\, \overline{([x], \eta)}, \quad \eta \in H^{\perp},
\]
where $F \subseteq G/H, {\rm card}(F) < \infty$, and $c_{[x]} \in \CC$, we have
\begin{multline}
\label{eq:ineq12}
     A\,\int_{H^{\perp}}\,|\Theta_F(\eta)|^2\,\Phi_{\Ccal}(f)(\eta)\,d\eta \leq
     \int_{H^{\perp}}\,|\Theta_F(\eta)|^2\,\Phi_{\Ccal}(f)(\eta)^2\,d\eta
\\
     \leq B\,\int_{H^{\perp}}\,|\Theta_F(\eta)|^2\,\Phi_{\Ccal}(f)(\eta)\,d\eta < \infty.
\end{multline}
\end{lem}

\begin{proof}$\,(\Rightarrow)$ Define
\[
    g_F = \sum_{[x]\in F}\,c_{[x]} (\tau_{[x],\Ccal}\,f) \in V_{\Ccal,f} \subseteq  L^2(G).
\]
By Lemma \ref{lem:53}, 
\begin{equation}
\label{eq:1}
    \norm{g_F}_{L^2(G)}^2 = \int_{H^{\perp}}\,|\Theta_F(\eta)|^2\,\Phi_{\Ccal}(f)(\eta)\,d\eta;
\end{equation}
and, by Lemma \ref{lem:54},
\begin{equation}
\label{eq:2}
    \sum_{[x]\in G/H}\, |\langle g_F, \tau_{[x],\Ccal}\,f \rangle_{L^2(G)}|^2 = \int_{H^{\perp}}\,|\Theta_F(\eta)|^2\,
     \Phi_{\Ccal}(f)(\eta)^2\,d\eta.
\end{equation}
With the frame assumption, and using equations \eqref{eq:1} and \eqref{eq:2}, we compute
\begin{align*}
    A\,\int_{H^{\perp}}\,|\Theta_F  (\eta)|^2  \,\Phi_{\Ccal}(f)(\eta)\,d\eta
    &= A\, \norm{g_F}_{L^2(G)}^2 
    \leq  \sum_{[x]\in G/H}\, |\langle g_F, \tau_{[x],\Ccal}\,f \rangle_{L^2(G)} |^2
\\
   &= \int_{H^{\perp}}\,|\Theta_F(\eta)|^2\,\Phi_{\Ccal}(f)(\eta)^2\,d\eta
   \leq  \sum_{[x]\in G/H}\, |\langle g_F, \tau_{[x],\Ccal}\,f \rangle_{L^2(G)}|^2
\\
     &\leq B\, \norm{g_F}_{L^2(G)}^2 = B\,\int_{H^{\perp}}\,
    |\Theta_F(\eta)|^2\,\Phi_{\Ccal}(f)(\eta)\,d\eta,
\end{align*}
and this is \eqref{eq:ineq12}.

$(\Leftarrow)$ Because of the equalities of Lemmas \ref{lem:53} and \ref{lem:54}, and
by the definition of $g_F$, we can rewrite 
our assumption \eqref{eq:ineq12} as
\begin{equation}
\label{eq:span}
      A\,\norm{g}_{L^2(G)}^2
     \leq  \sum_{[x]\in G/H}\, |\langle g, \tau_{[x],\Ccal}\,f \rangle_{L^2(G)}|^2 \leq B\,\norm{g}_{L^2(G)}^2
\end{equation} 
for all $g \in {\rm span}\{\tau_{[x],\Ccal}\,f) : [x] \in G/H \}$. We then apply Lemma \ref{lem:52}
directly to obtain \eqref{eq:span} for all $g \in V_{\Ccal,f}$, the desired conclusion.
\end{proof}


\subsection{Characterization of frames of translates}
\label{sec:characterization}

The basic idea behind our 
proof of Theorem \ref{thm:char} is to integrate over the compact group $H^{\perp}$ instead of the section
$[0, 1)^d$, and then to sum over the section $\Ccal$ instead of the discrete subgroup $\ZZ^d$. This involves
the design of the
correct definition of translation, that we did in Subsection \ref{sec:trans}.
Thus, we switch from defining $\Phi$ over a discrete subgroup,
to defining $\Phi_{\Ccal}$ over a set of coset representatives; and then we integrate over 
the compact group $H^{\perp}$ instead of
a set of 
coset representatives, which for the case of $\RR^d$ can be the set $[0,1)^d$. 
Thus, in the case of our Euclidean Theorem \ref{thm:benli}, we integrated over the
torus group $\RR^d/\ZZ^d$, thought of as  periodization of $[0,1)^d$, which could obfuscate the fact
that we were really integrating over the section $[0, 1)^d$.

\begin{thm}
\label{thm:char}
Let $G$ be a LCAG with compact open subgroup $H$, and let
 $\Ccal$ be a set of coset representatives of $\widehat{G}/H^{\perp}$. Let $f \in L^2(G)$.
The sequence, $\{\tau_{[x],\Ccal}\,f : [x] \in G/H \}$, is a frame for $V_{\Ccal,f}$ with
frame bounds $A$ and $B$
if and only if
\begin{equation}
\label{eq:bdprop}
     \exists\, A, B > 0 \; {\rm such}\, {\rm that} \quad A \leq \Phi_{\Ccal}(f) \leq B \; {\rm on}\;
     H^{\perp} \setminus N,
\end{equation}
where $N = \{ \eta \in H^{\perp} : \Phi_{\Ccal}(f)(\eta) = 0\}$ and $N$ is defined up to sets of measure $0$.
\end{thm}

\begin{proof}
$(\Leftarrow)$ Assume the inequalities \eqref{eq:bdprop}, and let
\[
    \Theta_F(\eta) = \sum_{[x]\in F}\,c_{[x]}\, \overline{([x], \eta)}, \quad \eta \in H^{\perp},
\]
where $F \subseteq G/H, {\rm card}(F) < \infty$, and $c_{[x]} \in \CC$. $\Theta_F$ is a 
trigonometric polynomial defined on $H^{\perp}$. We shall prove the inequalities
\eqref{eq:ineq12} as follows:
\begin{align*}
   A\,\int_{H^{\perp}}\,|\Theta_F(\eta)|^2\,\Phi_{\Ccal}(f)(\eta)\,d\eta
   &=
   A\,\int_{H^{\perp} \setminus N}\,|\Theta_F(\eta)|^2\,\Phi_{\Ccal}(f)(\eta)\,d\eta
\\
   &\leq
   \int_{H^{\perp} \setminus N}\,|\Theta_F(\eta)|^2\,\Phi_{\Ccal}(f)(\eta)^2\,d\eta
\\
   &\leq
   B  \int_{H^{\perp} \setminus N}\,|\Theta_F(\eta)|^2\,\Phi_{\Ccal}(f)(\eta)\,d\eta
   =
   B  \int_{H^{\perp} }\,|\Theta_F(\eta)|^2\,\Phi_{\Ccal}(f)(\eta)\,d\eta.
\end{align*}
Thus, we obtain the result by Lemma \ref{lem:55}.

$(\Rightarrow)$ {\it i.} Let $\{\tau_{[x],\Ccal}\,f : [x] \in G/H \}$, be a frame for $V_{\Ccal,f}$ with
frame bounds $A$ and $B$. Thus, by Proposition \ref{prop:bessel}, $\Phi_{\Ccal}(f) \in L^2(H^{\perp})$.
Assume $\Phi_{\Ccal}(f) < A$ on some set $E \subseteq H^{\perp} \setminus N$ for which 
$\nu_{H^{\perp}}(E)> 0$.
We shall show that the first inequality of
\eqref{eq:ineq12} fails, thereby obtaining the desired contradiction by Lemma \ref{lem:55}.
A similar contradiction will arise, in this case using the second inequality of \eqref{eq:ineq12}
and Lemma \ref{lem:55}, when we assume $\Phi_{\Ccal}(f) > B$
on some set $E \subseteq H^{\perp} \setminus N$ for which $\nu_{H^{\perp}}(E)> 0$, Thus, we obtain
\eqref{eq:bdprop}.

{\it ii.} By our assumption on $\Phi_{\Ccal}(f)$, we can choose $\Theta \in L^{\infty}(H^{\perp})$
such that $\Theta = 0$ off $E$, $|\Theta| > 0$ on $E$, and
\[
    A\, \int_{H^{\perp}}\,|\Theta(\eta)|^2\,\Phi_{\Ccal}(f)(\eta)\,d\eta
    >  \int_{H^{\perp}}\,|\Theta(\eta)|^2\,\Phi_{\Ccal}(f)(\eta)^2\,d\eta,
\]
e.g., take $\Theta = {\mathbb 1}_E$. Thus,
\begin{equation}
\label{eq:p}
    p = \int_{H^{\perp}}\,|\Theta(\eta)|^2\,\big( A\,\Phi_{\Ccal}(f)(\eta) - \Phi_{\Ccal}(f)(\eta)^2 \big) \,d\eta > 0.
\end{equation}
We shall find a trigonometric polynomial $\Theta_F$ on $H^{\perp}$ so that the strict inequality \eqref{eq:p}
is valid for $\Theta$ replaced by $ \Theta_F$. Thus, \eqref{eq:ineq12} fails in this case, and we obtain the
contradiction sought in part {\it i.}

Note that if $\Phi_{\Ccal}(f) \in L^2(H^{\perp}) \setminus L^{\infty}(H^{\perp})$, then
the strict inequality 
\eqref{eq:p} is still valid for $\Theta \in L^{\infty}(H^{\perp})$; and the plan to choose
$\Theta_F$ does not a priori require $\Phi_{\Ccal}(f) \in L^{\infty}(H^{\perp})$.

{\it iii.} Still only assuming $\Phi_{\Ccal}(f) \in L^{2}(H^{\perp})$
(and this is the case by the frame hypothesis and Proposition \ref{prop:bessel}),
we have for any $\Psi \in  L^{\infty}(H^{\perp})$ that
\begin{align}
\int_{H^{\perp}}\,|\Psi(\eta)|^2\, &
\big( A\,\Phi_{\Ccal}(f)(\eta) - \Phi_{\Ccal}(f)(\eta)^2 \big) \,d\eta 
\notag \\
\label{eq:decomp}
   &= \int_{H^{\perp}\setminus E}\,|\Psi(\eta)|^2\,\big( A\,\Phi_{\Ccal}(f)(\eta) - \Phi_{\Ccal}(f)(\eta)^2 \big) \,d\eta 
\\
& \phantom{blahbl} + \int_{E}\,|\Psi(\eta) -\Theta(\eta) + \Theta(\eta)|^2\,\Phi_{\Ccal}(f)(\eta)
    \big(A - \Phi_{\Ccal}(f)(\eta) \big) \,d\eta.
\notag
\end{align}
Since $A - \Phi_{\Ccal}(f) > 0$ on $E$ and $\Phi_{\Ccal}(f) > 0 \; {\rm a.e.}$ on
$H^{\perp} \setminus N$, we see that 
\[
   \Phi_{\Ccal}(f)\, \big(A - \Phi_{\Ccal}(f)\big) > 0\; {\rm a.e.} \, {\rm on} \, E; 
\]
and we consider $ \Phi_{\Ccal}(f)\, \big(A - \Phi_{\Ccal}(f)\big)$ as the weight for
a weighted $L^2$ Hilbert space on $E$. Thus, with $\Theta,\,\Psi \in L^{\infty}(H^{\perp})$,
we can obtain the estimate
\begin{align}
         \bigg(\int_{E}\,|\Psi(\eta) & -\Theta(\eta) + \Theta(\eta)|^2\,\Phi_{\Ccal}(f)(\eta)
    \big(A - \Phi_{\Ccal}(f)(\eta) \big) \,d\eta \bigg)^{1/2}
\notag \\
\label{eq:weighted}
    &\geq \bigg(\int_{E}\,|\Psi(\eta) -\Theta(\eta)|^2\,\Phi_{\Ccal}(f)(\eta)
    \big(A - \Phi_{\Ccal}(f)(\eta) \big) \,d\eta \bigg)^{1/2}
\\
     & \phantom{blahbl}  - \bigg(\int_{E}\,|\Theta(\eta)|^2\,\Phi_{\Ccal}(f)(\eta)
    \big(A - \Phi_{\Ccal}(f)(\eta) \big) \,d\eta \bigg)^{1/2}.
\notag
\end{align}
Combining \eqref{eq:decomp} and \eqref{eq:weighted}, we have
\begin{align}
\int_{H^{\perp}}\,|\Psi(\eta)|^2\, &
\big( A\,\Phi_{\Ccal}(f)(\eta) - \Phi_{\Ccal}(f)(\eta)^2 \big) \,d\eta
\notag \\
\label{eq:pp1/2}
&\geq \, p - 2p^{1/2}\, \bigg(\int_{E}\,|\Psi(\eta) -\Theta(\eta)|^2\,\Phi_{\Ccal}(f)(\eta)
    \big(A - \Phi_{\Ccal}(f)(\eta) \big) \,d\eta \bigg)^{1/2}.
\end{align}

{\it iv.} As we saw 
in Remark \ref{rem:ortho}, the trigonometric polynomials are dense in $L^2(H^{\perp})$.
We shall show how to choose such a trigonometric polynomial $\Psi$ so that
\begin{equation}
\label{eq:final}
    \int_{H^{\perp}}\,|\Psi(\eta)|^2\,\big( A\,\Phi_{\Ccal}(f)(\eta) - \Phi_{\Ccal}(f)(\eta)^2 \big) \,d\eta
    \geq \frac{p}{2} >0.
\end{equation}
This is the strict inequality promised in part {\it ii.} yielding the 
desired contradiction set out in part {\it i.}

We begin by using our assumption that $\Phi_{\Ccal}(f) < A$ on $E \subseteq H^{\perp} \setminus N$
in the following way. Let $t = \Phi_{\Ccal}(f)$ so that $t \in (0,A)$, and consider $r(t) =
t^2 - At + A^2/4$ on $[0,A]$. Then by calculus $r(t) > 0$ on $[0,A]$ except at $A/2$ where it
vanishes. Consequently, $|A\,\Phi_{\Ccal}(f) - \Phi_{\Ccal}(f)^2| \leq A^2/4$ on $E$,
and so
\[
   2\,p^{1/2}\, \bigg(\int_{E}\,|\Psi(\eta) -\Theta(\eta)|^2\,\Phi_{\Ccal}(f)(\eta)
    \big(A - \Phi_{\Ccal}(f)(\eta) \big) \,d\eta \bigg)^{1/2} 
    \leq Ap^{1/2} \norm{\Psi - \Theta}_{L^2(H^{\perp})}
\]
for all $\Psi \in L^{\infty}(H^{\perp})$. Now choose $\Psi \in L^{\infty}(H^{\perp})$
to be a trigonometric polynomial for which $\norm{\Psi - \Theta}_{L^2(H^{\perp})}
\leq p^{1/2}/(2A)$.
This combined with \eqref{eq:pp1/2} gives \eqref{eq:final}.

The choice of polynomial is easier to achieve directly from part {\it iii.}
in the explicit case that $\Theta = {\mathbb 1}_E$.

\end{proof}


\section{Examples}
\label{sec:ex}

As number theoretic background, recall the classification theorem of 
non-discrete locally compact fields K.
{\it If ${\rm char}(K) =0$, then $K$ is $\RR, \, \CC$, or a finite extension of $\Qp$; 
and if ${\rm char}(K) = p >0$, then $K$ is ultrametric and isomorphic to the
field of formal power series over a finite field}, see \cite{RamVal1999}, Section 4.2.
Except for $\RR$ and $\CC$, these fields $K_v$ are non-Archimedean. This infers 
that the ring of integers $\Ocal_v \subseteq K_v$ is a compact open subgroup under addition,
where $v$ designates the discrete absolute value $|\cdot|_v$ used to define the topology on $K_v$;
e.g., see Example \ref{ex:QpZp} and \cite{serr1979}, \cite{RamVal1999}, \cite{robe2000}. 
We refer to each $K_v$ as a locally compact local field or simply \emph{local field}.

\begin{example}
\label{ex:twoH}
Let $G=K_v$ be a finite extension of the field $\Qp$, under addition, 
where $p\geq 3$ is an odd prime,
and let $H=\Ocal_v$ be the ring of integers of $K_v$. In particular,
we could have $G  = \Qp$ and $H=\Zp$. All such fields have the property
that $\widehat{G}$ is (non-canonically) isomorphic to $G$, and that
for each $a\in G$, we have $p^n a \in H$ for some $n\geq 0$.

Fix $a\in G$ with $a\not\in H$, and define $f\in L^2(G)$ by
\[ 
      f = {\mathbb 1}_H + {\mathbb 1}_{a+H}. 
\]
According to Example~\ref{ex:transchi}, for each $m\in\ZZ$, we have
\[ \tau_{[ma],\Ccal}f = (a,\sigma_0)^m\big( {\mathbb 1}_{ma+H} + {\mathbb 1}_{(m+1)a+H} \big). \]
Thus, ${\mathbb 1}_{a+H} + {\mathbb 1}_{2a+H}\in V_{\Ccal,f}$, and hence
${\mathbb 1}_H - {\mathbb 1}_{2a+H}\in V_{\Ccal,f}$.
Since the order of $a$ in $G/H$ is odd, summing translates of 
${\mathbb 1}_H - {\mathbb 1}_{2a+H}$ gives ${\mathbb 1}_H-{\mathbb 1}_{a+H}\in V_{\Ccal,f}$,
and hence ${\mathbb 1}_H\in V_{\Ccal,f}$. Therefore,
\[ V_{\Ccal,f} = \{g\in L^2(G) : g \text{ is constant on cosets } x+H \} \]
is generated by functions of the form ${\mathbb 1}_{x+H}$.

A simple computation shows
\[ \widehat{f}(\gamma) = \big( 1 + \overline{(a,\gamma)} \big) {\mathbb 1}_{H^{\perp}} (\gamma), \]
and, hence, for $\eta\in H^{\perp}$, we have
\begin{equation}
\label{eq:PhitwoH}
\Phi_{\Ccal}(f)(\eta) = \sum_{\sigma\in\Ccal} \big| \hat{f}(\eta + \sigma) \big|^2
=\big| 1 + \overline{(a,\eta+\sigma_0)} \big|^2 .
\end{equation}
The order of $[a]\in G/H$ is $p^n$ for some $n\geq 1$, and so
$\overline{(a,\eta+\sigma_0)}$ can take the value of any $p^n$-root of unity.
Substituting all such values into equation~\eqref{eq:PhitwoH}, we obtain
\[ 2  - 2\cos(\pi/p^n) \leq \Phi_{\Ccal}(f) \leq 4,\]
since the $p^n$-root of unity closest to $-1$ is $-e^{\pi i/p^n}$,
and the furthest is $1$.
Thus, Theorem~\ref{thm:char} allows us to conclude that $\{\tau_{[x],\Ccal} f : [x]\in G/H \}$
is a frame for $V_{\Ccal,f}$ with frame constants
$A=2  - 2\cos(\pi/p^n)$ and $B=4$.
\end{example}

\begin{example}
\label{ex:twoH2}
Let $G=\QQ_2$ and $H=\ZZ_2$, or, more generally, let $G=K_v$ be a finite extension of $\QQ_2$,
with $H=\Ocal_v$ the associated ring of integers. As in Example~\ref{ex:twoH}, we have that
$G$ and $\widehat{G}$ are isomorphic; and for each $a\in G$, we have $2^n a \in H$ for some $n\geq 0$.

Also, as in Example~\ref{ex:twoH}, fix $a\in G$ with $a\not\in H$, and define
$f={\mathbb 1}_H + {\mathbb 1}_{a+H} \in L^2(G)$. Once again, we have
\[ \tau_{[ma],\Ccal}f = (a,\sigma_0)^m\big( {\mathbb 1}_{ma+H} + {\mathbb 1}_{(m+1)a+H} \big)
\quad \text{for all } m\in\ZZ . \]
This time, however, since the order of $[a] \in G/H$ is even,
${\mathbb 1}_H - {\mathbb 1}_{a+H}$ does \emph{not} belong to $V_{\Ccal,f}$,
and, in fact,
\[ V_{\Ccal,f} = \{g\in L^2(G) : g \text{ is constant on sets of the form } (x+H) \cup (x+a+H) \} \]
is generated by functions of the form ${\mathbb 1}_{x+H}+{\mathbb 1}_{x+a+H}$.

By the same reasoning as in Example~\ref{ex:twoH}, $\Phi_{\Ccal}(f)$ is
given by the formula of equation~\eqref{eq:PhitwoH}. On the other hand, since
the order of $[a] \in G/H$ is $2^n$ for some $n\geq 1$, there is a set $N\subseteq H^{\perp}$
of positive measure where $(a,\eta+\sigma_0)=-1$, and hence where $\Phi_{\Ccal}(f)(\eta)=0$.
For $\eta\in H^{\perp}\setminus N$, however,
$(a,\eta+\sigma_0)$ is a $2^n$-roots of unity other than $-1$.
The closest such root to $-1$ is $-e^{2\pi i / 2^n}$, and the furthest is $1$.
Thus, we have
\[ 2  - 2\cos(\pi/2^{n-1}) \leq \Phi_{\Ccal}(f) \leq 4.\]
By Theorem~\ref{thm:char}, then, $\{\tau_{[x],\Ccal} f : [x]\in G/H \}$
is a frame for $V_{\Ccal,f}$ with frame constants
$A=2  - 2\cos(\pi/2^{n-1})$ and $B=4$.
\end{example}

\begin{example}
\label{ex:cH}
Let $G=K_v$ be a locally compact local field with ring of integers $H=\Ocal_v$,
and let $c\in H$ have the property $0<|c|_v < 1$, so that $cH$ is a compact open subgroup of $H$
of some index $n\geq 2$. (For instance, with $G=\Qp$, $H=\Zp$, and $c=p^m$,
the index of $p^m \Zp$ in $\Zp$ is $p^m$.)
Define $f=\sqrt{n} {\mathbb 1}_{cH}\in L^2(G)$. Then
\begin{equation}
\label{eq:cHfhat}
\widehat{f}(\gamma) = \sqrt{n}\int_{cH} \overline{(x,\gamma)} \, dx
= \frac{1}{\sqrt{n}} \int_{H} \overline{(cx,\gamma)} \, dx
= \frac{1}{\sqrt{n}} \int_{H} \overline{(x,c\gamma)} \, dx
= \frac{1}{\sqrt{n}} {\mathbb 1}_{c^{-1} H^{\perp}}(\gamma) .
\end{equation}
We also have $[c^{-1} H^{\perp} : H^{\perp}]=n$, and hence we can write
$\Ccal\cap c^{-1} H^{\perp} = \{\sigma_0, \sigma_1,\ldots,\sigma_{n-1}\}$.
A simple computation shows that 
\[ 
      \forall b \in G, \quad \tau_{[b],\Ccal}f(x) = \frac{1}{\sqrt{n}} \bigg( \sum_{i=0}^{n-1} 
      (x,\sigma_i) \bigg) {\mathbb 1}_{b+H}(x), 
\]
which is constant on all sets of the form $x+cH$. It is not obvious from this formula
that the translates of $f$ form a tight frame, but they do. In fact,
by equation~\eqref{eq:cHfhat}, we have
\[ 
\Phi_{\Ccal}(f)(\eta) = \sum_{\sigma\in\Ccal} \big| \hat{f}(\eta + \sigma) \big|^2
= \sum_{i=0}^{n-1} \Big( \frac{1}{\sqrt{n}} \Big)^2 = 1. 
\]
Then, by Theorem~\ref{thm:char}, we conclude that $\{\tau_{[x],\Ccal} f : [x]\in G/H \}$
is a FUNTF for $V_{\Ccal,f}$.
\end{example}

\begin{example}
\label{ex:cH2}
With $G$, $H$, $c$, and $n\geq 2$ as in Example~\ref{ex:cH},
define $f={\mathbb 1}_{c^{-1}H}\in L^2(G)$. Then,
\[ 
\widehat{f}(\gamma) = \int_{c^{-1}H} \overline{(x,\gamma)} \, dx
=n \int_{H} \overline{(c^{-1}x,\gamma)} \, dx
= n \int_{H} \overline{(x,c^{-1}\gamma)} \, dx
= n {\mathbb 1}_{c H^{\perp}}(\gamma); 
\]
and, hence, for any $b\in G$, we have
\[ \widehat{\tau_{[b],\Ccal}f}(\gamma)
= n \, \overline{(b,\eta_{\gamma})} {\mathbb 1}_{c H^{\perp}}(\gamma)
= n \, \overline{(b,\gamma - \sigma_0)} {\mathbb 1}_{c H^{\perp}}(\gamma),
= n \, (b,\sigma_0) \overline{(b,\gamma)} {\mathbb 1}_{c H^{\perp}}(\gamma), \]
since any $\gamma\in c H^{\perp}\subseteq H^{\perp}$ has $\sigma_\gamma=\sigma_0$.
Thus,
\begin{align*}
\tau_{[b],\Ccal}f(x)
&= n (b,\sigma_0) \int_{cH^\perp} (x-b,\gamma) \, d\gamma
= (b,\sigma_0) \int_{H^\perp} (x-b,c\gamma) \, d\gamma
\\
&= (b,\sigma_0) \int_{H^\perp} \big( c(x-b),\gamma\big) \, d\gamma
= (b,\sigma_0) {\mathbb 1}_{H}\big( c(x-b) \big)
= (b,\sigma_0) {\mathbb 1}_{b+c^{-1}H}(x).
\end{align*}
Therefore,
\[ V_{\Ccal,f} = \{ g\in L^2(G) : g \text{ is constant on cosets } x+ c^{-1} H \}, \]
and
\[ \Phi_{\Ccal}(f)(\eta) = \sum_{\sigma\in\Ccal} \big| \hat{f}(\eta + \sigma) \big|^2
= {\mathbb 1}_{cH^{\perp}}(\eta+\sigma_0)
= {\mathbb 1}_{-\sigma_0 + cH^{\perp}}(\eta). \]
Note that $\Phi_{\Ccal}(f)$ is zero on the set $N=H^{\perp}\setminus (-\sigma_0 + cH^{\perp})$,
which has measure $(n-1)/n$. However, $\Phi_{\Ccal}(f)$ is $1$ on $H^{\perp}\setminus N$,
and so $\{\tau_{[x],\Ccal} f : [x]\in G/H \}$ is a FUNTF for $V_{\Ccal,f}$
by Theorem~\ref{thm:char}.
\end{example}


\bibliographystyle{amsplain}

\bibliography{2018-03-05JBbib}

\providecommand{\bysame}{\leavevmode\hbox to3em{\hrulefill}\thinspace}
\providecommand{\MR}{\relax\ifhmode\unskip\space\fi MR }
\providecommand{\MRhref}[2]{%
  \href{http://www.ams.org/mathscinet-getitem?mr=#1}{#2}
}
\providecommand{\href}[2]{#2}
\begin{thebibliography}{10}

\bibitem{BarHerPat2015}
D.~Barbieri, Eugenio Hern{\'a}ndez, and Victoria Paternostro, \emph{The {Z}ak
  transform and the structure of spaces invariant by the action of an {LCA}
  group}, J. Funct. Analysis \textbf{269} (2015), no.~5, 1327--1358.

\bibitem{BarHerPat2016}
D.~Barbieri, Eugenio Hern{\' a}ndez, and Victoria Paternostro, \emph{Group
  {R}iesz and frame sequences: The bracket and the {G}ramian}, Collectanea
  Mathematica (2016).

\bibitem{bene1975}
John~J. Benedetto, \emph{Spectral {S}ynthesis}, Academic Press, Inc., New York,
  1975.

\bibitem{bene1994}
\bysame, \emph{Frame decompositions, sampling, and uncertainty principle
  inequalities}, Wavelets: {M}athematics and {A}pplications (John~J. Benedetto
  and Michael~W. Frazier, eds.), CRC Press, Boca Raton, FL, 1994, pp.~247--304.

\bibitem{bene1997}
John~J. Benedetto, \emph{Harmonic {A}nalysis and {A}pplications}, Studies in
  Advanced Mathematics, CRC Press, Boca Raton, FL, 1997. \MR{MR1400886
  (97m:42001)}

\bibitem{BenBen2004a}
John~J. Benedetto and Robert~L. Benedetto, \emph{A wavelet theory for local
  fields and related groups}, J. Geom. Anal. \textbf{14} (2004), 423--456.

\bibitem{BenCza2009}
John~J. Benedetto and Wojciech Czaja, \emph{Integration and {M}odern
  {A}nalysis}, Birkh\"auser Advanced Texts, Springer-Birkh\"auser, New York,
  2009.

\bibitem{BenLi-1993}
John~J. Benedetto and Shidong Li, \emph{Multiresolution analysis frames with
  applictions}, IEEE ICASSP (International Conference on Acoustics and Signal
  Processing), Minneapolis (1993).

\bibitem{BenWal1994}
John~J. Benedetto and David Walnut, \emph{Gabor frames for {L}$^{2}$ and
  related spaces}, Wavelets: {M}athematics and {A}pplications, edited by J.J.
  Benedetto and M. Frazier, CRC (1994), 97--162.

\bibitem{bene2003}
Robert~L. Benedetto, \emph{Examples of wavelets for local fields}, Wavelets,
  frames and operator theory, Contemp. Math., vol. 345, Amer. Math. Soc.,
  Providence, RI, 2004, pp.~27--47.

\bibitem{BowRos2015}
Marcin Bownik and Kenneth~A. Ross, \emph{The structure of translation-spaces on
  locally compact {A}belian groups}, J. Fourier Analysis and Appl. \textbf{21}
  (2015), no.~4, 849--884.

\bibitem{CabPat2010}
Carlos Cabrelli and Victoria Paternostro, \emph{Shift--invariant spaces on
  {LCA} groups}, Journal of Functional Analysis \textbf{258} (2010),
  2034--2059.

\bibitem{chri2016}
Ole Christensen, \emph{An {I}ntroduction to {F}rames and {R}iesz {B}ases, 2nd
  edition}, Springer-Birkh{\"a}user, New York, 2016 (2003).

\bibitem{daub1992}
Ingrid Daubechies, \emph{Ten {L}ectures on {W}avelets}, CBMS-NSF Regional
  Conference Series in Applied Mathematics, Society for Industrial and Applied
  Mathematics, 1992.

\bibitem{GohGol1981}
Israel Gohberg and Seymour Goldberg, \emph{Basic {o}perator {t}heory},
  Birkh\"auser Boston, Mass., 1981. \MR{MR632943 (83b:47001)}

\bibitem{GolTou2008}
R.~A.~Kamyabi Gol and R.~Raisi Tousi, \emph{The structure of shift invariant
  spaces on a locally compact abelian group}, J.~Math. Analysis Appl.
  \textbf{340} (2008), 219--225.

\bibitem{GroStr2007}
Karlheinz Gr{\"o}chenig and Thomas Strohmer, \emph{Pseudodifferential operators
  on locally compact abelian groups and {S}j{\"o}strand's symbol class}, J.
  reine angew Math. (Crelle) \textbf{613} (2007), 121--146.

\bibitem{HerSikWeiWil2010}
Eugenio Hern{\'a}ndez, H.~Siki{\'c}, Guido~L. Weiss, and Edward~N. Wilson,
  \emph{Cyclic subspaces for unitary representations of {LCA} groups;
  generalized {Z}ak transforms}, Colloq Math. \textbf{118} (2010), 313--332.

\bibitem{HewRos1963}
Edwin Hewitt and Kenneth~A. Ross, \emph{Abstract {H}armonic {A}nalysis,
  {V}olume {I}}, Springer-Verlag, New York, 1963.

\bibitem{HewRos1970}
\bysame, \emph{Abstract {H}armonic {A}nalysis, {V}olume {II}}, Springer-Verlag,
  New York, 1970.

\bibitem{kobl1984}
Neal Koblitz, \emph{$p$-adic {N}umbers, $p$-adic {A}nalysis, and
  {Z}eta-{F}unctions}, second ed., Springer-Verlag, New York, 1984.

\bibitem{matu2018}
Ewa Matusiak, \emph{Frames of translates on model sets}, arXiv: 1801. 05213v2
  [math.FA] 23 Jan. 2018 (2018).

\bibitem{pont1966}
Lev~Lemenovich Pontryagin, \emph{Topological {G}roups, 2nd edition}, Gordon and
  Breach, Science Publishers, Inc., New York, 1966, Translated from the Russian
  by Arlen Brown.

\bibitem{RamVal1999}
Dinakar Ramakrishnan and Robert~J. Valenza, \emph{{F}ourier {A}nalysis on
  {N}umber {F}ields}, Springer-Verlag, New York, 1999. \MR{2000d:11002}

\bibitem{reit1968}
Hans Reiter, \emph{Classical {H}armonic {A}nalysis and {L}ocally {C}ompact
  {G}roups}, Oxford University Press, 1968.

\bibitem{robe2000}
Alain~M. Robert, \emph{A {C}ourse in p-adic {A}nalysis}, Springer-Verlag, New
  York, 2000.

\bibitem{rudi1962}
Walter Rudin, \emph{{F}ourier {A}nalysis on {G}roups}, John Wiley and Sons, New
  York, 1962.

\bibitem{serr1979}
Jean-Pierre Serre, \emph{Local {F}ields}, Springer-Verlag, New York, 1979,
  Translated by Marvin J.~Greenberg.

\bibitem{stro1999}
Thomas Strohmer, \emph{Rates of convergence for the approximation of dual
  shift-invariant systems}, J. Fourier Analysis Appl. \textbf{5} (1999), no.~6,
  599--615.

\end{thebibliography}

\end{document}